\documentclass[12pt,a4paper,reqno]{amsart}

\usepackage{amsmath,enumitem, comment, microtype}
\usepackage{graphicx}

\usepackage{amssymb,amsthm}
\usepackage{mathrsfs}
\usepackage{bbm}

\def\R{\mathbb R}
\def\T{\mathbb T}
\def\E{\mathcal E}
\def\L{\mathcal L}
\DeclareMathOperator{\hdim}{dim_{\mathrm H}}
\DeclareMathOperator{\tr}{tr}

\newtheorem{lem}{Lemma}[section]

\newtheorem{nt}[lem]{Notation}

\newtheorem{theorem}[lem]{Theorem}
\newtheorem{lemma}[lem]{Lemma}

\theoremstyle{remark}
\newtheorem{remark}[lem]{\bf Remark}
\numberwithin{equation}{section}

\allowdisplaybreaks[2]

\title{Hausdorff dimension of recurrence sets}

\author{Zhangnan Hu*}\thanks{* Corresponding author} 
\address{Z.-N.~Hu, 
             School of Mathematics, South China University of Technology, Guangzhou, 510641, P. R. China}
\email{hnlgdxhzn@163.com}

\author{Tomas Persson}
\address{T.~Persson, Centre for Mathematical Sciences, Lund
  University, Box~118, 221~00~Lund, Sweden}
\email{tomasp@maths.lth.se}

\date{\today}

\begin{document}

\begin{abstract}
  We consider linear mappings on the $d$-dimensional torus,
  defined by $T(x) = Ax \pmod 1$, where $A$ is an invertible
  $d \times d$ integer matrix, with no eigenvalues on the unit
  circle. In the case $d = 2$ and $\det A = \pm 1$, we give a
  formula for the Hausdorff dimension of the set
  \[
    \{ \, x \in \mathbb{T}^d : d (T^n (x), x) < e^{- \alpha n}
    \text{ for infinitely many } n \, \}.
  \]
%  and we conjecture a similar formula for $d > 2$.
\end{abstract}

\maketitle

% I think the following results also hold for some integer matrix
% A with \det A=1 and an eigenvalues $\lambda>1$, but here we
% consider a special case first.

\section{Introduction}

Let $(X, \mathscr{B},T, \mu,d)$ be a metric measure preserving
system (m.m.p.s.).
% The distribution of the orbit of a point in $X$ is an important
% topic in ergodic theory and has been studied by many
% authors. See, for example,
% \cite{AP,Bo,BW,FLL,FanScT13,fur,HLSW}.
If $(X,d)$ is a separable metric space, then the well-known
Poincar\'e recurrence theorem shows that $\mu$-a.e.\ $x\in X$ is
recurrent, that is
\[
  \liminf_{n\to\infty}d(T^nx,x)=0.
\]
It tells us that for $\mu$-almost every $x\in X$, the orbit
returns to a sequence of shrinking targets of the initial point
infinitely many times. However, it shows nothing about the speed
at which the orbit can return to the initial point or the
shrinking targets of the the initial point. Boshernitzan
\cite{Bo} investigated the rate of recurrence for general
systems.

\begin{theorem}[\cite{Bo}]
  Let $(X, \mathscr{B},T, \mu,d)$ be a m.m.p.s. Assume that for
  some $\tau>0$, the $\tau$-dimensional Hausdorff measure
  $\mathcal{H}^{\tau}$ of $X$ is $\sigma$-finite. Then for
  $\mu$-a.e.\ $x\in X$,
  \[
    \liminf_{n\to\infty}n^{\frac{1}{\tau}}d(T^nx,x)< \infty.
  \]
  Futhermore, if $\mathcal{H}^{\tau} (X)=0$, then for $\mu$-almost
  every $x\in X$,
  \[
    \liminf_{n\to\infty}n^{\frac{1}{\tau}}d(T^nx,x)=0.
  \] 
\end{theorem}

Later, Barreira and Saussol \cite{BS} related the rate of
recurrence to the lower pointwise dimension.

\begin{theorem}[\cite{BS}]
  If $T \colon X \to X$ is a Borel measurable map on a measurable
  subset $X \subset \mathbb{R}^m$, and $\mu$ is a $T$-invariant
  Borel probability measure on $X$, then for $\mu$-almost every
  $x \in X$, we have
  \[\liminf_{n\to\infty}n^{1/\tau} d(T^nx,x)=0\]
  for any $\tau>\underline{d}_{\mu} (x)$, where 
  \[
    \underline{d}_{\mu} (x)= \liminf_{r\to0} \frac{\log
      \mu(B(x,r))}{\log r}.
  \]
\end{theorem}

Hence a natural question is how large is the set of recurrent
points when the rate of recurrence is replaced by a general
function. More precisely, Let $(X, \mathscr{B},T, \mu,d)$ be a
m.m.p.s.\ and $r_n(x)$ be some positive function on
$\mathbb{N} \times X$. Define the recurrence set as
\[
  E = E(r_n) = \{\, x\in X : T^nx\in B(x, r_n(x))~\text{ for
    infinitely many } n\ge1 \, \}.
\]
Tan and Wang \cite{tanwang} calculated the Hausdorff dimension of
$E(r_n)$ when $T$ is the $\beta$-trans\-formation with
$\beta>1$. Later, Seuret and Wang \cite{sw15} proved a similar
result for conformal iterated function systems. Chang, Wu and Wu
\cite{cww} considered the recurrence set on a self-similar set
with the strong separation condition. Then Baker and Farmer
\cite{bafa} generalised their results to finite conformal
iterated function systems. Hussein, Li, Simmons and Wang
\cite{hlsw} showed that the measure of $E(r_n)$ obeys a
zero--full law for some conformal and expanding systems.
Kirsebom, Kunde and Persson \cite{kkp} investigated the measure
of $E(r_n)$ for a class of mixing interval maps and some linear
maps on tori.

The recurrence set is a limsup set which often has a large
intersection property, originally introduced by Falconer
\cite{fallip}. Given $s\in(0,m]$, he defined
$\mathcal{G}^s(\mathbb{R}^m)$ to be the class of all $G_{\delta}$
sets $F$ in $\mathbb{R}^d$ such that the Hausdorff dimension any
set in $\mathcal{G}^s(\R^m)$ is at least $s$, and closed under
similarity transformations and countable intersections. To define
the corresponding class of sets on the $d$-dimensional torus
$\mathbb{T}^d$ is straightforward.

Persson and Reeve \cite{perssonreeve} used Riesz potentials to
determine if a limsup set belongs to the class
$\mathcal{G}^s(\T^m)$. The following lemma is important for the
proof of our results.
% for some $M_n>0$, and $\mathcal{E}_n^i$ are ellipsoids with
% semi-axes $r_n\frac{\lambda^n-1}{{S_k'}S_k}$ and
% $r_n\frac{1- \lambda^{-n}}{{S_k'}S_k}$.  Let
% $(\T^2, \mathcal{B}, \L, T)$ be a measure preserving
% system, % with the usual quotient distance $d$,%with a Euclidean metric $d$, the usual quotient distance d%$\mathcal{B}$ is a Borel $\sigma$-algebra of $X$,
% and $\L$ is Lebesgue measure on $\T^2$. Let $\{r_n\}_{n\ge1}$
% be a sequence of positive real numbers.  Define the recurrence
% set as
%\[E= \{x\in\T^2\colon T^nx\in B(x,r_n)~i.m. ~n\ge1\}.\]
\begin{lemma}[Lemma~2.1 in \cite{per22}. See also \cite{perssonreeve}]\label{lemma:persson}
  Let $E_n$ be open sets in $\T^m$ and let $\mu_n$ be measures
  with $\mu_n( \T^m\setminus E_n) = 0$. If there is a constant
  $C$ such that
  \[
    C^{-1} \le \liminf_{n\to\infty} \mu_n(B) \le
    \limsup_{n\to\infty} \mu_n(B) \le C
  \] 
  for any ball B, and
  \[
    \iint |x-y|^{-s}d\mu_n(x)d\mu_n(y)<C
  \]
  for all $n$, then
  $\limsup_{n\to\infty}E_n \in \mathcal{G}^s(\T^m)$, and in
  particular we have
  $\hdim (\limsup_{n\to\infty}E_n ) \ge s$.
\end{lemma}

Motivated by above results, we focus on the Hausdorff dimension
of $E(r_n)$ when $T$ is a linear mapping on $\T^2$ and the rate
$r_n$ does not depend on the initial point, that is, the set
\[
  E = \{\, x\in\T^2 : T^n(x) \in B(x,r_n) \text{ for infinitely
    many } n\ge1 \, \}.
\]

\begin{theorem}
  \label{main}
  Let $A$ be a $2 \times 2$ integer matrix with $|\!\det A|=1$ and
  an eigenvalue $|\lambda|>1$. Let $T(x)=Ax \pmod 1$, and for
  $n\ge1$, $r_n=e^{- \alpha n}$, $\alpha>0$. Then
  % Let
  % $s_0= \min \{\frac{2\log\lambda}{\alpha+ \log\lambda},
  % \frac{\log\lambda}{\alpha} \}$, then\\
  \[
    \hdim E = s_0,
  \]
  where
  \[
    s_0= \min
    \Bigl\{\frac{2\log|\lambda|}{\alpha + \log|\lambda|},
    \frac{\log|\lambda|}{\alpha} \Bigr\}.
  \]
  Moreover, for $\alpha>0$, we have $E\in\mathcal{G}^{s_0} (\T^2).$ 
\end{theorem}

\begin{figure}
  \includegraphics[scale=1]{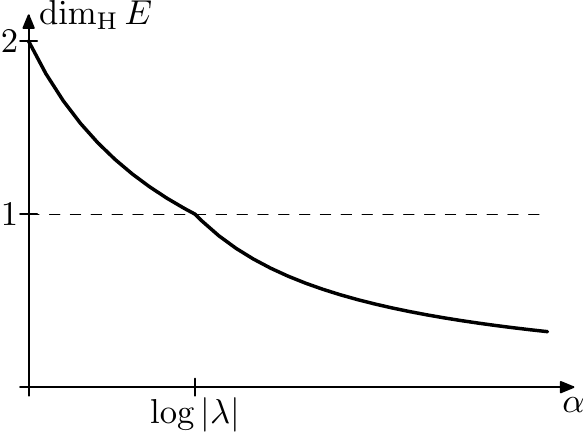}
  \caption{The graph of $\hdim E$ as a function of $\alpha$, with $A =\bigl[\protect \begin{smallmatrix} 2 & 1 \protect\\ 1 &1 \protect\end{smallmatrix}   \bigr].$}
\end{figure}

He and Liao \cite[Theorem~1.7]{HeLiao} gave a formula for the
dimension of $E$ when $A$ is a diagonal matrix, not necessarily
integer, and with all diagonal elements of modulus larger than
1. Our result seems to be the first result of this type when
there is both contraction and expansion present.

\begin{remark}
  For a general non-increasing sequence of positive real numbers
  $\{r_n\}_{n\ge1}$, let
  \[
    \alpha= \liminf_{n\to\infty} \frac{- \log r_n}{n}.
  \]
  Then Theorem~\ref{main} also holds.
\end{remark}

%\begin{remark}
%  Suppose that $A$ is a $d \times d$ integer matrix with eigenvalues $\lambda_1, \ldots, \lambda_d$ such that  $|\lambda_1| \geq \ldots \geq|\lambda_k| > 1 > |\lambda_{k+1}| \geq \ldots \geq|\lambda_d| > 0$. In this case, a similar covering  argument to the one we use in Section~\ref{updim} gives that
%  \begin{multline*}
%    \hdim E \leq \min \Bigl \{\, \frac{\log |\lambda_1 \ldots
%      \lambda_{l-1}| + (d+1-l) \log |\lambda_l|}{\alpha + \log
  %    |\lambda_l|}, \\ \frac{\log |\lambda_1 \ldots
  %    \lambda_k|}{\alpha} : l = 1, \ldots, k \, \Bigr\}.
%  \end{multline*}
%  We know that this upper bound is not always an equality by \cite[Theorem~1.7]{HeLiao}.
%\end{remark}

Let $A$ and $\{r_n\}_{n\ge1}$ be as in
Theorem~\ref{main}. %= \begin{bmatrix} a&b\\c&d\end{bmatrix}$ be
                    %an integer matrix with $\det A=1$ and an
                    %eigenvalue $\lambda>1$. For $n\ge1$, let
                    %$r_n=e^{- \alpha n}$, where $\alpha>0$.
Rewrite the recurrence set as
\[
  E = \{\, x\in\T^2 : A^nx \pmod 1 \in B(x,r_n) \text{ i.o.} \,
  \},
\]
and let
\begin{align*}
  \E_n &:= \{\,x\in\T^2 : A^nx \pmod 1 \in B(x,r_n) \, \} \\%=
                                %\bigcup_{i=1}^{M_n}
                                %\mathcal{E}_n^i, \]
       &\phantom{:}= \{x\in\T^2 : (A^n-I)x \pmod 1 \in B(0,r_n) \, \},
\end{align*}
where
$I$ is the identity matrix. Then $E= \limsup_{n\to\infty} \E_n$.

The paper is organised as follows.  In next section, we will
discuss the periodic points
\[
  \{\, x \in\T^2 : (A^n-I)x \pmod 1 =0 \, \},
\]
which is crucial to our proof since %$\E_n$ 
to understand
$\E_n$, we need to understand the distribution of these periodic
points. The proof of our main result is divided into two
parts. In Section 3, we give the upper bound on the Hausdorff
dimension of $E$. The last section is dedicated to prove that
$E$ has a large intersection property, which gives the lower
bound on the Hausdorff dimension of $E$.

Without loss of generality, we only prove Theorem~\ref{main} for
$\det A=1$ and $\lambda>1$, since for $\lambda<-1$ or
$\det A=-1$, we can consider $A^2$ instead of $A$, whose
eigenvalues are $\lambda^2>1$ and $\lambda^{-2}$, and
$\det A^2=1$. Hence we omit the proof for other cases.

\begin{nt}
  Write $f_n\lesssim g_n$, $n\in\mathbb{N}$, if there is an
  absolute constant $0<c< \infty$ such that for all
  $n\in\mathbb{N}$, $f_n\le cg_n$.  If $f_n\lesssim g_n$ and
  $g_n\lesssim f_n$ for $n\in\mathbb{N}$, then we denote
  $f_n\asymp g_n$.
\end{nt}

\section{Periodic points} % when $\det A=1$}

Let
$A= \bigl[ \begin{smallmatrix} a & b \\ c & d \end{smallmatrix}
\bigr]$ be an integer matrix with $\det A=1$ and eigenvalue $\lambda>1$.
% is different from that when $\det A=-1$, we will discuss these
% two cases separately.  Let
% $A= \begin{bmatrix} a&b\\c&d\end{bmatrix}$ be an integer matrix
% with $det A=1$ and eigenvalue $\lambda$. Without losing
% generality, we assume that $\lambda>1$, since we can use
% similar method to prove the results also hold when
% $\lambda<-1$.
Then
\[
  \begin{bmatrix} 1\\ \frac{\lambda-a}{b} \end{bmatrix} \quad{\rm
    and} \quad  \begin{bmatrix} 1\\
    \frac{\lambda^{-1}-a}{b} \end{bmatrix}
\]
are eigenvectors with eigenvalues $\lambda$ and
$\lambda^{-1}$. We can diagonalise $A$ as $A=TDT^{-1}$, where
\[
  T= \begin{bmatrix} 1 & 1 \\
    \frac{\lambda-a}{b} & \frac{\lambda^{-1}-a}{b}
  \end{bmatrix}, \ T^{-1} =
  \frac{b\lambda}{1- \lambda^2}
  \begin{bmatrix}
    \frac{\lambda^{-1}-a}{b} & -1 \\
    - \frac{\lambda-a}{b} & 1 \rule{0pt}{13pt}
  \end{bmatrix} \text{ and } D= \begin{bmatrix}
    \lambda&0\\0&\lambda^{-1} \end{bmatrix}.
\]
Therefore, for $n\ge1$,
\begin{align*}
  A^n &= TD^nT^{-1} \\%= \frac{b\lambda}{1-
  % \lambda^2} \begin{bmatrix}
  %   1&1\\\frac{\lambda-a}{b}
  %   &\frac{\lambda^{-1}-a}{b} \end{bmatrix} \begin{bmatrix}
  %   \lambda^n&0\\0&\lambda^{-n} \end{bmatrix} \begin{bmatrix}
  %   \frac{\lambda^{-1}-a}{b}&-1\\-
  %   \frac{\lambda-a}{b} &1\end{bmatrix} \\
  % &= \frac{b\lambda}{1- \lambda^2}
  % \begin{bmatrix} \frac{(\lambda^{-1}-a) \lambda^n- (\lambda-a)
  %   \lambda^{-n}}{b}& - \lambda^n+ \lambda^{-n} \\
  %   \frac{(\lambda^{-1}-a) (\lambda-a) \lambda^n- (\lambda-a)
  %   (\lambda^{-1}-a) \lambda^{-n}}{b^2}&\frac{(\lambda^{-1}-a)
  %   \lambda^{-n}- (\lambda-a) \lambda^{n}}{b} \end{bmatrix} \\
    &=
    \begin{bmatrix}
      \frac{(\lambda^{-1}-a) \lambda^{n+1} - (\lambda-a)
        \lambda^{-n+1}}{1 - \lambda^2} &
      \frac{b( - \lambda^{n+1}+ \lambda^{-n+1})}{1- \lambda^2} \\
      \frac{(\lambda^{-1}-a) (\lambda-a) \lambda^{n+1} -
        (\lambda-a) (\lambda^{-1}-a) \lambda^{-n+1}}{b(1 -
        \lambda^2)} & \frac{(\lambda^{-1}-a) \lambda^{-n+1} -
        (\lambda-a) \lambda^{n+1}}{1- \lambda^2} \rule{0pt}{14pt}
    \end{bmatrix}.
\end{align*}
It follows that
\begin{multline*}
  A^n-I \\
  = \begin{bmatrix} \frac{(\lambda^{-1}-a) \lambda^{n+1} -
      (\lambda-a) \lambda^{-n+1}}{1- \lambda^2} -1
    &\frac{b(- \lambda^{n+1} + \lambda^{-n+1})}{1 - \lambda^2} \\
    \frac{(\lambda^{-1}-a) (\lambda-a) \lambda^{n+1} -
      (\lambda-a) (\lambda^{-1}-a) \lambda^{-n+1}}{b(1 -
      \lambda^2)} &\frac{(\lambda^{-1}-a) \lambda^{-n+1} -
      (\lambda-a) \lambda^{n+1}}{1- \lambda^2}-1 \rule{0pt}{14pt}
  \end{bmatrix}.
\end{multline*}
Put
\[
  H_n := \det(A^n-I). %2- \lambda^n- \lambda^{-n}:=H_n$.
\]
 
From the equations above, we have
\begin{multline*}
  (A^n-I)^{-1} \\ \!\! = {\textstyle \frac{1}{H_n}} \!
  \begin{bmatrix}
    \frac{(\lambda^{-1} - a) \lambda^{-n+1} - (\lambda-a)
      \lambda^{n+1}}{1 - \lambda^2} - 1&\frac{b(\lambda^{n+1} -
      \lambda^{-n+1})}{1 - \lambda^2} \\
    \frac{(a - \lambda^{-1}) (\lambda-a) \lambda^{n+1} +
      (\lambda-a) (\lambda^{-1}-a) \lambda^{-n+1}}{b(1-
      \lambda^2)} & \frac{(\lambda^{-1}-a) \lambda^{n+1} -
      (\lambda-a) \lambda^{-n+1}}{1 - \lambda^2} - 1
  \end{bmatrix}.
\end{multline*}
% If $x$ is the periodic point, $x$ can be expressed as
% \[x= \frac{m}{H_n} \begin{bmatrix} \frac{(\sqrt{5}+1)
%     \lambda^n+ (\sqrt{5}-1) \lambda^{-n}}{2\sqrt{5}}-1\\-
%     \frac{\lambda^n- \lambda^{-n}}{\sqrt{5}} \end{bmatrix}+
%   \frac{n}{H_n} \begin{bmatrix} - \frac{\lambda^n-
%     \lambda^{-n}}{\sqrt{5}} \\\frac{(\sqrt{5}-1) \lambda^n+
%     (\sqrt{5}+1) \lambda^{-n}}{2\sqrt{5}}-1\end{bmatrix}, \]
% where $m,~n$ are integers.  By the definition of $\E_n$,
We observe that $\E_n$ consists of elliptical discs
$\{\E_n^i\}_i$ with semi-axes comparable to
$r_n\frac{\lambda^n-1}{|H_n|}$ and $r_n\frac{1-
  \lambda^{-n}}{|H_n|}$, whose centres are periodic points with
period $n$. If $A$ is symmetric, then the semi-axes of $\{\E_n^i\}_i$ equal to $r_n\frac{\lambda^n-1}{|H_n|}$ and $r_n\frac{1-
  \lambda^{-n}}{|H_n|}$, since the eigenvectors of $A$ are orthogonal. 
  
Recall that a point $x\in\T^2$ is called a \emph{periodic point}
with period $n$ if
\begin{equation} \label{equation:periodic}
  (A^n-I)x~(\rm mod~1)=0.
\end{equation}
In this section, we investigate the periodic points in order to
understand the distribution of the elliptical discs
$\{\E_n^i\}_i$. %From our observation, the distribution of
                %periodic points is different when n is taken by
                %different value.

\subsection{The number of periodic points}

\begin{lemma}
  For every $n\ge1$, $H_n=2- \lambda^n- \lambda^{-n}$, and $H_n$
  is an integer.
\end{lemma}

\begin{proof}
  For $n\ge1$, we have
  \begin{equation*}
    \begin{split}
      H_n& = \det (A^n-I) =
      \det(D^n-I) \\%(\frac{(\lambda^{-1}-a) \lambda^{n+1}-
                    %(\lambda-a) \lambda^{-n+1}}{1- \lambda^2}-1)
                    %(\frac{(\lambda^{-1}-a) \lambda^{-n+1}-
                    %(\lambda-a) \lambda^{n+1}}{1- \lambda^2}-1)
                    %\\
      &= (\lambda^n-1) (\lambda^{-n}-1)%\quad- \frac{b( -
                                       %\lambda^{n+1}+
                                       %\lambda^{-n+1})}{1-
                                       %\lambda^2}
                                       %\frac{(\lambda^{-1}-a)
                                       %(\lambda-a)
                                       %\lambda^{n+1}-
                                       %(\lambda-a)
                                       %(\lambda^{-1}-a)
                                       %\lambda^{-n+1}}{b(1-
                                       %\lambda^2)} \\
      =2- \lambda^n- \lambda^{-n},
    \end{split}
  \end{equation*}
  Since $A^n-I$ is an integer matrix, $H_n$ is an integer.
\end{proof}

\begin{lemma}
  \label{number:periodic}
  The number of periodic points with period $n$ is $|H_n|$, that
  is
  \[
    %G_n\colon = 
    \# \{ \, x \in \T^2 : (A^n-I)x \pmod 1 = 0
    \, \} = |H_n|.
  \]
\end{lemma}
%How many integer lattices are in $(A^n-I) \mathbb{T}^2$? 
We will use the following theorem to prove Lemma~\ref{number:periodic}.

\begin{theorem}[Pick's theorem \cite{pick}]
  Suppose that a polygon has integer coordinates for all of its
  vertices. Let $i$ be the number of integer points interior to
  the polygon, and let $b$ be the number of integer points on its
  boundary (including both vertices and points along the
  sides). Then the area (square units) of the polygon is
  \[
    i+{\frac{b}{2}}-1.
  \]
\end{theorem}

\begin{proof}[Proof of Lemma~\ref{number:periodic}]
  Notice that the number of periodic points is the same as the number
  of the solutions of \eqref{equation:periodic}, which equals to
  the number of the integer lattice points in
  $(A^n-I) \mathbb{T}^2$.

  We assume that
  $A^n-I= \bigl[ \begin{smallmatrix} a_n & b_n \\ c_n &
    d_n \end{smallmatrix} \bigr]$. Then $(A^n-I) \mathbb{T}^2$ is
  the parallelogram with coordinates
  \[
    (0,0), \,(a_n,c_n), \,(b_n,d_n), \,(a_n+b_n,c_n+d_n).
  \]
  Denote the segment linking $(0,0)$ to $(a_n,c_n)$ by $L_1$, the
  segment linking $(0,0)$ to $(b_n,d_n)$ by $L_2$, the segment
  linking $(b_n,d_n)$ to $(a_n+b_n,c_n+d_n)$ by $L_3$, and the segment
  linking $(a_n,c_n)$ to $(a_n+b_n,c_n+d_n)$ by $L_4$.

  Here we suppose that $i$ is the number of integer points
  interior of $(A^n-I) \mathbb{T}^2$, and $N_k$ is the number of
  integer points on $L_k$, $k=1,2,3,4$. Then the number of
  integer points on the boundary of $(A^n-I) \mathbb{T}^2$ is
  $\sum_kN_k-4$. Note that $N_1=N_3$ and $N_2=N_4$.  Hence the
  number of integer lattices in $(A^n-I) \mathbb{T}^2$ is
  $i+N_1+N_2-3$.

  By Pick's Theorem, we have
  \[
    \mathcal{L} ((A^n-I) \mathbb{T}^2)=i+
    (\sum_kN_k-4)/2-1=i+N_1+N_2-3.
  \]
  Since
  $\mathcal{L} ((A^n-I) \mathbb{T}^2)=|a_nd_n-c_nb_n|=|\!\det
  (A^n-I) |$, we have the number of integer lattices in
  $(A^n-I) \mathbb{T}^2$ equals to $|\!\det(A^n-I) |$, that is,
  $|H_n|$.
  % Now we show the claim. We only prove that $N_1=N_3$. If
  % $(m,n)$ is on $L_1$, then $cm=an$, hence
  % $\frac{c}{a} (m+b-b)+d=n+d$ which means that $(m+b,n+d)$ is
  % on $L_3$. Similarly, If $(m,n)$ is on $L_3$, then $(m-b,n-d)$
  % is on $L_1$. Therefore $N_1=N_3$
\end{proof}

\subsection{Periodic points when $\boldsymbol{n}$ is odd}

To obtain the lower bound on $\hdim E$, it suffices to study the
distribution of periodic points when $n$ is odd.  In the
following, we only consider the case when $n$ is odd. Assume that
$n=2k+1$. Put
\[
  S_k = \lambda^k + \ldots + \lambda + 1 + \lambda^{-1} + \ldots
  + \lambda^{-k} = \frac{\lambda^{-k} - \lambda^{k+1}}{1-
    \lambda}.
\]
Notice that
$S_k = 1 + \tr D + \ldots + \tr D^k = 1 + \tr A + \ldots + \tr
A^k$, where $\tr D $ denotes the trace of $D$. Hence $S_k$ is an
integer.

\begin{lem}
  \label{lem:Hn}
%For every $n\ge1$, $H_n=2- \lambda^n- \lambda^{-n}$. 
  For every $n=2k+1$, $H_n=- (\tr A - 2) S_k^2.$
\end{lem}

\begin{proof}
  For every $k\ge0$, we have
  \begin{align}
    \nonumber
    S_k^2
    & = \frac{(\lambda^{k+1} - \lambda^{-k})^2}{(\lambda-1)^2}
      = \frac{\lambda^{2k+2} + \lambda^{-2k} - 2
      \lambda}{(\lambda-1)^2} \\
    \label{skhn}
    & = \frac{\lambda}{(\lambda - 1)^2} (\lambda^{2k+1} +
      \lambda^{-2k-1}-2) = -
      \frac{\lambda}{(\lambda-1)^2}H_{2k+1}.
  \end{align}
  For $k=0$, we have $1=S_0^2=- \frac{\lambda}{(\lambda-1)^2}H_1$,
  also $H_1=2- \lambda- \lambda^{-1}=2- \tr A$.
  % In fact,  the eigenvalue $\lambda>1$, which can be expressed as
  % \[\lambda= \frac{a+d+ \sqrt{a^2+4bc-2ad+d^2}}{2}, \]
  % since $det A=ad-bc=1$, then
  % \[\lambda= \frac{a+d+ \sqrt{(a+d)^2-4}}{2}.\]
  So 
  \begin{equation}
    \label{det1}
    1= \frac{\lambda}{(\lambda-1)^2} (\tr A - 2).
    % \begin{split}
    %   (\lambda-1)^2&= \Big(\frac{a+d-2+ \sqrt{(a+d)^2-4}}{2}
    %   \Big)^2= \frac{(a+d-2) (a+d+ \sqrt{(a+d)^2-4})}{2} \\
    %   &= (\tr A - 2) \lambda.
    % \end{split}
  \end{equation}
  It follows from \eqref{skhn} and \eqref{det1} that
  $H_n=- (\tr A - 2)S_k^2$.
\end{proof}
% \begin{remark}
%   Note that $\det A=1$ is neccessary, otherwise
%   $\frac{(\lambda-1)^2}{\lambda}$ is not an integer. From
%   $\lambda>1$, we have $a+d-2\ne0$.
%\end{remark}

%\noindent
%\rule{\textwidth}{2pt}

%\textbf{I could not understand the proof of Lemma~\ref{per:contain}, and it is a bit unclear to me how you mean that Lemma~\ref{per:contained} follows from Lemma~\ref{per:contain}. Is it by counting as in the remark below?  I have added a proof of Lemma~\ref{per:contained}, which makes the text between the line above and the line below  unnecessary. But maybe we can complete the proof of Lemma~\ref{per:contain}?  (Just for completeness, it would be nice to know exactly what the periodic points are.)}

Since we are assuming that $\lambda > 1$, we have $\tr A -
2>0$. Let $S_k'= (\tr A - 2) S_k$. By the same assumption, we
have $S_k' > 0$. Write
\[
  \mathscr{P}_o =
  \Bigl\{\Bigl(\frac{m}{S_k'}, \frac{j}{S_k'} \Bigr) : 0\le m,j\le
  S_k'-1\Bigr\}.
\]
  
\begin{lemma} \label{per:contain}
  For $n=2k+1$, $k\ge0$, the periodic points with period $n$ are
  contained in $\mathscr{P}_o$.
\end{lemma}

%Notice that $(1- \lambda)^2= \lambda$, hence
%\[S_k^2= \lambda^{-2k-1} (1-2\lambda^n+ \lambda^{2n})= \lambda^{-n}-2+ \lambda^n=-H_n.\]
\begin{proof}
  We prove that $A_{2k+1} := (\tr A - 2) S_k (A^{2k+1}-I)^{-1}$
  is an integer matrix.  If so, any periodic point
  $x=(A^{2k+1}-I)^{-1}[m_1 \ m_2]^t$ for $m_1,m_2\in\mathbb{Z}$,
  can be rewritten as
  \[x=\frac{1}{S_k'}A_{2k+1}[m_1 \ m_2]^t.\]
  Since $A_{2k+1}[m_1 \ m_2]^t$ is an integer vector, this implies that $x\in \mathscr{P}_o$.
  
  %Suppose that
  %$x = \frac{1}{S_k'} [x_1 \ x_2]^t$, where $x_1$ and $x_2$ are integers, then $(A^n - I)x \in \mathbb{Z}^2$ if and only if $[x_1 \ x_2]^t \in A_{2k+1} \mathbb{Z}^2 \subset\mathbb{Z}^2$. We would need that $A_{2k+1} \mathbb{Z}^2 = \mathbb{Z}^2$, but that does not need to be the case since $\det A_{2k+1} = \tr A - 2$ which could be different from $\pm 1$.

  Since $H_n = -(\tr A - 2) S_k^2$ by Lemma~\ref{lem:Hn}, we have that
  \begin{multline*}
      A_{2k+1} \\
      = \! \textstyle{\frac{-1}{S_k}} \!
      \begin{bmatrix}
        \frac{(\lambda^{-1}-a) \lambda^{-n+1} - (\lambda-a)
          \lambda^{n+1}}{1- \lambda^2} - 1
        & \frac{b( \lambda^{n+1} - \lambda^{-n+1})}{1 -
          \lambda^2} \\
        \frac{- (\lambda^{-1}-a) (\lambda-a) \lambda^{n+1} +
          (\lambda-a) (\lambda^{-1} - a) \lambda^{-n+1}}{b(1 -
          \lambda^2)}
        &\frac{(\lambda^{-1} - a) \lambda^{n+1} - (\lambda-a)
          \lambda^{-n+1}}{1- \lambda^2} {\scriptstyle-1}
      \end{bmatrix}.
  \end{multline*}
  Put
  \begin{equation}\label{rk}
    \begin{split}
      R_k&= \lambda^k - \lambda^{k-1} + \ldots + (-1)^{k} +
      (-1)^{k-1} \lambda^{-1} + \ldots -
      \lambda^{-k+1} + \lambda^{-k} \\
      &= \frac{\lambda^{k+1}+ \lambda^{-k}}{1+ \lambda}.
    \end{split}
  \end{equation}
  Since
  $R_k= (-1)^{k} + (-1)^{k-1} \tr A + \ldots - \tr A^{k-1} + \tr
  A^k$, we have that $R_k$ is an integer.
  
    We have
  \begin{align*}
    \tr A_{2k+1} &= \frac{\lambda^{2k + 1} - 1}{S_k} +
                   \frac{\lambda^{-2k - 1} - 1}{S_k} \\
                 &= (\lambda - 1) \lambda^k - \lambda^{-2k - 1}
                   \frac{\lambda^{2k + 1} - 1}{S_k} \\
                 &= (\lambda - 1) (\lambda^k - \lambda^{-k - 1}).
  \end{align*}
  Moreover,
  \begin{align*}
    \det A_{2k+1} &=
                    \frac{(\lambda^{2k+1}-1)(\lambda^{-2k-1}-1)}{S_k^2} \\
                  &= (\lambda - 1)^2 \frac{2 - \lambda^{2k+1} -
                    \lambda^{-2k-1}}{(\lambda^{k+1} - \lambda^{-k})^2}
                    \lambda^k 
                    = - \frac{(\lambda - 1)^2}{\lambda}.                  
  \end{align*}
  Hence, $\tr A = - \det (A_{2k+1}) S_k$. Since $\det A_{2k+1}$
  does not depend on $k$, we have
  $\det A_{2k+1} = \det A_1 = \det (A-I) = 2 - \tr A$, which is
  an integer. It follows that $\tr A_{2k+1} = (2 - \tr A) S_k$ is
  an integer as well.

  Finally, we let $e_1 = [\begin{matrix} 1 & 0 \end{matrix}]^t$ and
  $e_2 = [\begin{matrix} 0 & 1 \end{matrix}]^t$, compute
  $e_2^t A_{2k+1} e_2$, $e_2^t A_{2k+1} e_1$ and $e_1^t A_{2k+1} e_2$ and show that they
  are integers.  After some simplifications, we obtain that
  \begin{align}
    \nonumber
    e_2^tA_{2k+1}e_2
    &= - \frac{1}{S_k} \Bigl(
      \frac{(\lambda^{-1}-a) \lambda^{n+1} -
      (\lambda-a) \lambda^{-n+1}}{1 - \lambda^2} - 1 \Bigr) \\
    \nonumber
    &= \frac{1}{1 + \lambda} \Bigl( \frac{\lambda^{2k+1} -
      \lambda^{-2k+1} + \lambda^2-1}{\lambda^{1+k} -
      \lambda^{-k}} + a\frac{\lambda^{-2k} -
      \lambda^{2k+2}}{\lambda^{1+k} - \lambda^{-k}} \Bigr) \\
    \nonumber
    &= \frac{1}{1+ \lambda} \Bigl( \frac{(\lambda^{1+k} -
      \lambda^{-k}) (\lambda^{k} +
      \lambda^{-k+1})}{\lambda^{1+k} - \lambda^{-k}} -a
      \frac{\lambda^{2k + 2} - \lambda^{-2k}}{\lambda^{1 + k} -
      \lambda^{-k}} \Bigr) \\
    \nonumber
    &= \frac{1}{1 + \lambda} \Bigl( \lambda^{k} + \lambda^{-k+1}
      - a(\lambda^{k+1} + \lambda^{-k})   \Bigr) \\
    &=R_{k-1}-aR_k,
      \label{p1}
  \end{align}
  and
  \begin{equation}
    \label{p2}
    \begin{split}
      e_1^tA_{2k+1}e_2&=- \frac{1}{S_k} \frac{b( \lambda^{n+1}-
        \lambda^{-n+1})}{1- \lambda^2} \\
      &= \frac{\lambda-1}{\lambda^{-k}- \lambda^{k+1}} \frac{b(
        \lambda^{2k+2}- \lambda^{-2k})}{1- \lambda^2} \\
      &= \frac{b}{1+ \lambda} \frac{\lambda^{-2k}-
        \lambda^{2k+2}}{\lambda^{-k}- \lambda^{k+1}} \\
      &= \frac{b}{1+ \lambda} (\lambda^{-k}+ \lambda^{k+1})=bR_k.
    \end{split}
  \end{equation}
  Also
  \begin{align}
    \nonumber
    e_2^tA_{2k+1}
    &e_1
      = \frac{1}{S_k} \frac{(\lambda^{-1}-a)
      (\lambda-a) \lambda^{n+1}- (\lambda-a) (\lambda^{-1}-a)
      \lambda^{-n+1}}{b(1- \lambda^2)} \\
    \nonumber
    &= \frac{1- \lambda}{\lambda^{-k}- \lambda^{k+1}}
      \frac{(\lambda^{-1}-a) (\lambda-a) \lambda^{n+1}-
      (\lambda-a) (\lambda^{-1}-a) \lambda^{-n+1}}{b(1-
      \lambda^2)} \\
    \nonumber
    &= \frac{1}{\lambda^{-k}- \lambda^{k+1}}
      \frac{(\lambda^{-1}-a) (\lambda-a) (\lambda^{2k+2}-
      \lambda^{-2k})}{b(1+ \lambda)} \\ 
    &=- \frac{(\lambda^{-1}-a) (\lambda-a) (\lambda^{k+1}+
      \lambda^{-k})}{b(1+ \lambda)}. 
    \label{p3}
  \end{align}

  Notice that $e_2^tA_{1}e_1=c$, that is
  \begin{equation*}
    \begin{split}
      c & = \frac{(\lambda^{-1}-a) (\lambda-a) \lambda^{2} -
        (\lambda-a) (\lambda^{-1}-a)}{b(1 - \lambda^2)} \\
      & = \frac{(\lambda^{-1}-a) (\lambda-a) (\lambda^{2} -
        1)}{b(1 - \lambda^2)} \\
      & = - \frac{(\lambda^{-1}-a) (\lambda-a)}{b}.
    \end{split}
  \end{equation*}
  Then \eqref{p3} can be rewritten as
  \begin{equation}
    \label{p4}
    e_2^tA_{2k+1}e_1=cR_k.   
  \end{equation}
  From \eqref{p1}, \eqref{p2} and \eqref{p4}, we have that the
  matrix elements $e_2^tA_{2k+1}e_1$, $e_1^tA_{2k+1}e_2$,
  $e_2^tA_{2k+1}e_2$ are all integers.   Since $\tr A_{2k+1}$ is an integer, the lower right element of $A_{2k+1}$ must be an integer
  as well.
 % This shows that the upper row in $A_{2k+1}$ contains only integers. Since $\tr A_{2k+1}$ is an integer, the lower right element of $A_{2k+1}$ must be an integer as well. Finally, the lower left element $A_{2k+1}$ is also an integer, either by computation, or using that the above statements are equaly true for $A$ replaced by $A^t$.
 % Note that
 % \begin{equation*}
%    \begin{split}
 %     \tr A_{2k+1}&= (\lambda-1) (\lambda^k- \lambda^{-k-1}) \\
 %     &= \frac{(\lambda-1)^2}{\lambda}S_k.
    %\end{split}
 % \end{equation*}
  %From \eqref{det1}, we have $\tr A_{2k+1}= (a+d-2)S_k$, which is an integer. Therefore
%  \begin{align*}
  %  e_1^tA_{2k+1}e_1
  %  &= - \frac{1}{S_k} \frac{(\lambda^{-1}-a) \lambda^{-n+1} -
   %   (\lambda-a) \lambda^{n+1}}{1- \lambda^2}-1 \\
  %  &= \tr A_{2k+1}-e_2^tA_{2k+1}e_2
 % \end{align*}
%  is also an integer.
  % \be
  % \begin{split}
  %   det A_{2k+1}&= - \frac{1}{S_k^2}H_{2k+1} \\
  %   &= \frac{(\lambda-1)^2}{\lambda}.
  % \end{split}
  % \ee
 % Hence $A_{2k+1}$ is an integer matrix.
\end{proof}

%\noindent
%\rule{\textwidth}{2pt}

\begin{lemma}
  \label{per:contained}
  For $n=2k+1$, $k\ge0$, the points
  $\bigl\{\, \bigl( \frac{m}{S_k}, \frac{j}{S_k} \bigr) : m,j=0,
  \ldots, S_k-1 \,\bigr\}$ are periodic points.
\end{lemma}

\begin{proof}
  %Put
  %\[
   % R_k = \lambda^k - \lambda^{k-1} + \ldots + \lambda^{-k} =
  %  \frac{\lambda^{k+1} + \lambda^{-k}}{\lambda + 1}.
  %\]
 % Then $R_k$ is an integer since $R_k = \tr A^k - \ldots$

  We prove that $A_{2k + 1}' := \frac{1}{S_k} (A^{2k + 1} - I)$ is
  an integer matrix, with $\det A_{2k+1}' = 2 - \tr A$. This
  implies that if $x = \frac{1}{S_k} [x_1\ x_2]^t$ where $x_1$
  and $x_2$ are integers, then $(A^{2k+1} - I) x$ is an integer
  matrix, and hence $T^{2k+1} (x) = x$.

  Notice that \[
    e_1^t A_{2k+1}' e_1 = -e_2^t A_{2k+1} e_2,\quad  e_2^t A_{2k+1}' e_2 = -e_1^t A_{2k+1} e_1,
  \]
  \[
    e_1^t A_{2k+1}' e_2 = e_1^t A_{2k+1} e_2,\quad   e_2^t A_{2k+1}' e_1 = e_2^t A_{2k+1} e_1,
  \]
  where $A_{2k+1}=(trA-2)S_k(A^{2k+1}-I)^{-1}$. From the proof of
  Lemma~\ref{per:contain}, we have that the matrix elements of
  $A_{2k+1}'$ are all integers.
  %which is an integer. Similarly, we obtain that
%  \[
 %   e_1^t A_{2k+1} e_2 = b \frac{\lambda^{k+1} +
  %    \lambda^{-k}}{\lambda + 1} = b R_k,
 % \]
 % which is also an integer. This shows that the upper row in $A_{2k+1}$ contains only integers. Since $\tr A_{2k+1}$ is an integer, the lower right element of $A_{2k+1}$ must be an integeras well. Finally, the lower left element $A_{2k+1}$ is also an  integer, either by computation, or using that the above  statements are equaly true for $A$ replaced by $A^t$.
\iffalse
  \[
    e_1^t A_{2k+1}' e_1 = - \frac{\lambda^k +
      \lambda^{-k+1}}{\lambda + 1} + a \frac{\lambda^{k+1} +
      \lambda^{-k}}{\lambda + 1} = - R_{k-1} + a R_k,
  \]
  which is an integer. Similarly, we obtain that
  \[
    e_1^t A_{2k+1} e_2 = b \frac{\lambda^{k+1} +
      \lambda^{-k}}{\lambda + 1} = b R_k,
  \]
  which is also an integer. This shows that the upper row in
  $A_{2k+1}$ contains only integers. Since $\tr A_{2k+1}$ is an
  integer, the lower right element of $A_{2k+1}$ must be an integer
  as well. Finally, the lower left element $A_{2k+1}$ is also an
  integer, either by computation, or using that the above
  statements are equaly true for $A$ replaced by $A^t$.
  \fi
\end{proof}

\begin{remark}
  The number of periodic points found in Lemma~\ref{per:contained}
  is clearly $S_k^2$. This is enough for our needs, but we note
  that by Lemma~\ref{number:periodic}, the number of periodic
  points are $|H_{2k+1}|$ and by \eqref{skhn},
  $S_k^2 = - \frac{\lambda}{(\lambda-1)^2} H_{2k+1}$. Since
  $\frac{\lambda}{(\lambda-1)^2} = \frac{1}{\tr A - 2}$, we
  therefore have
  \[
    S_k^2 = - \frac{1}{\tr A - 2} H_{2k+1}.
  \]
  Hence the periodic points found in Lemma~\ref{per:contained}
  are all the periodic points if and only if $\tr A - 2 = \pm 1$.
\end{remark}

% For $n=2k+1$, $k\ge1$, let
%\be\label{E_n}
%E_n= \bigcup_{i=1}^{N_n} \E_{n}^i\subset \E_n,
%\ee
%here $N_n=S_k^2$,  and $\E_{n}^i$ is the ellipsoid with center $(\frac{m_i}{S_k}, \frac{j_i}{S_k} )$, and semi-axes $r_n\frac{\lambda^n-1}{|H_n|}$ and $r_n\frac{1- \lambda^{-n}}{|H_n|}$. Then $\limsup_{k\to\infty}E_{2k+1} \subset E$.
\begin{remark}
\textbf{When $\boldsymbol{n}$ is even}, that is, $n=2k$, $k\ge1$,  the periodic points with period $n$ are
  contained in
  \[
    \mathscr{P}_e = \Bigl\{ \Bigl( \frac{m}{g_k}, \frac{j}{g_k}
    \Bigr) : 0 \le m,j \le g_k - 1 \, \Bigr\},
  \]
  where $g_k= \sqrt{(a+d)^2-4} (\lambda^k- \lambda^{-k})$.
\end{remark}

\begin{nt} \label{ellipsoid} Let
  $A = \bigl[ \begin{smallmatrix} a & b \\ c &
    d \end{smallmatrix} \bigr]$ be an integer matrix with
  $\det A=1$ and an eigenvalue $\lambda>1$. From
  Lemma~\ref{number:periodic}, given $n\ge1$, the number of
  periodic points with \eqref{equation:periodic} is
  $|H_n|$. Write
  \[
    \lambda_{n,1}=r_n\frac{\lambda^n-1}{|H_n|} = \frac{r_n}{1 -
      \lambda^{-n}}, \quad \lambda_{n,2}=r_n\frac{1-
      \lambda^{-n}}{|H_n|} = \frac{r_n}{\lambda^n - 1}.
  \]
  Then $\E_n$ consists of the elliptical discs
  $\{\E_n^i\}_{i=1}^{|H_n|}$, and for $i\ge1$, $\E_n^i$ contains
  a parallelogram $E_n^i$ with vertices
  \begin{align*}
    & x_{n,i} + \begin{bmatrix} 1 & \frac{\lambda-a}{b}
      \end{bmatrix}^t \lambda_{n,2},
    & x_{n,i} - \begin{bmatrix} 1 &
      \frac{\lambda-a}{b} \end{bmatrix}^t \lambda_{n,2}, \\
    & x_{n,i} + \begin{bmatrix} 1 &
      \frac{\lambda^{-1}-a}{b} \end{bmatrix}^t \lambda_{n,1},
    & x_{n,i} - \begin{bmatrix} 1 &
      \frac{\lambda^{-1}-a}{b} \end{bmatrix}^t \lambda_{n,1},
  \end{align*}
  whose centre $x_{n,i}$ is a periodic point satisfying
  \eqref{equation:periodic}. Also the lengths of the diagonal
  lines of $E_n^i$ are $2\lambda_{n,1}$ and $2\lambda_{n,2}$.
    \end{nt}
% we can use similar method for proving Theorem~\ref{main} when $\det A=1$, $\lambda<-1$, or $\det A=-1,~|\lambda|>1$. 
%Without losing generality, from now on, we always assume that $\lambda>1$, since one could use similar method to prove Theorem~\ref{main} holds when $\lambda<-1$. 

\section{The upper bound on $\hdim E$} \label{updim}

In this section, we will give the upper bound on $\hdim E$. 

\begin{lem}
  \label{upperdimension}
  Let $A$ be a $2 \times 2$ integer matrix with $\det A=1$ and an
  eigenvalue $\lambda>1$. Let $T(x) = Ax \pmod 1$, and for
  $n\ge1$, $r_n=e^{- \alpha n}$, $\alpha\ge0$. Then
%Let $s_0= \min \{\frac{2\log\lambda}{\alpha+ \log\lambda}, \frac{\log\lambda}{\alpha} \}$, then\\
  \[\hdim E\le s_0, \]
  where
  \[
    s_0= \min \Bigl\{\frac{2\log\lambda}{\alpha+ \log\lambda},
    \frac{\log\lambda}{\alpha} \Bigr\}.
  \]
\end{lem}

\begin{proof}
  We observe that for any $n\ge1$,
  $E\subset \bigcup_{n\ge k} \E_n$.  Notice that $\E_n$ consists
  of $|H_n|$ elliptical discs with same shape. There are two
  natural coverings of $\E_n$.

  Case 1:\, For $\alpha>0$, the sequence $\{r_n\}_{n\ge1}$ is a
  sequence of positive real numbers which decreases to 0. Then
  for any $\delta>0$, there is some $k_0\ge1$ such that for any
  $k\ge k_0$, we have $r_k< \delta$. Hence
  \begin{equation}
    \label{naturalcovering}
    \mathcal{H}_{\delta}^s(E) \le \mathcal{H}_{\delta}^s
    (\bigcup_{n\ge k_0} \E_{n}) \le \sum_{n\ge k_0}
    \sum_{i=1}^{|H_n|} \mathcal{H}_{\delta}^s (\E_{n,i}).
  \end{equation}
  % here $\E_{n,i}$ is the ellipsoid with centre $x_i$, semi-axes
  % $\lambda_{1,n}:=r_n\frac{\lambda^n-1}{|H_n|}$ and
  % $\lambda_{2,n}:=r_n\frac{1- \lambda^{-n}}{|H_n|}$. %Notice
  % that for every $x\in\mathcal{P}_o^{2k+1}$ and
  % $y\in\mathcal{P}_e^{2k}$, the shapes of $\E_{k,x}$ and
  % $\E_{k,y}$ are same. 

  Since $\E_{n,i} \subset B(x_i, \lambda_{n,1})$, then
  \eqref{naturalcovering} can be rewritten as
  \begin{equation*}
    \begin{split}
      \mathcal{H}_{\delta}^s(E)&\le\sum_{n\ge k_0}
      \sum_{i=1}^{|H_n|} \mathcal{H}_{\delta}^s(B(x_i,
      \lambda_{n,1})) \\ 
      &\lesssim \sum_{n\ge k_0}r_n^s\lambda^n= \sum_{n\ge
        k_0}e^{n(\log\lambda- \alpha s)}.
    \end{split}
  \end{equation*}
  Therefore $s>\frac{\log\lambda}{\alpha},$ we have
  $\mathcal{H}^s(E)< \infty$, which implies that
  $\hdim E\le \frac{\log\lambda}{\alpha}$.

  Case 2:\, For $\alpha>0$, we use some squares with length
  $2\lambda_{n,2}$ to cover $\E_{n,i}$, and denote these squares
  by $\{S_{i,j} \}_j$. We suppose that the centre of $S_{i,j}$ is
  $x_{i,j}$, and
  \[
    \Gamma_i := \#\{S_{i,j} \}
    \asymp\frac{\lambda_{n,1}}{\lambda_{n,2}}=
    \frac{\lambda^n-1}{1- \lambda^{-n}} \asymp \lambda^n.
  \]
  % Note that
  % $\E_{n,x} \subset \bigcup_{i\in A_x}S(x_i, \lambda_{2,n})$,
  % here $S(x_i, \lambda_{2,n})$ is the square with center $x_i$
  % and length $2\lambda_{2,n}$, and
  % $\#A_x= \frac{\lambda_{1,n}}{\lambda_{2,n}}= \frac{\lambda^n-1}{1- \lambda^{-n}}$,
  For any $\delta>0$, there is some $n_0\ge1$ such that for any
  $n\ge n_0$, we have $2\sqrt{2} \lambda_{n,2}< \delta$. Then  
  \begin{equation*}
    \begin{split}
      \mathcal{H}_{\delta}^s(E)&\le \mathcal{H}_{\delta}^s
      (\bigcup_{n\ge k_0} \E_{n}) \le \sum_{n\ge k_0}
      \sum_{i=1}^{|H_n|}
      \sum_{j=1}^{\Gamma_i} \mathcal{H}_{\delta}^s(S_{i,j}) \\
      &\lesssim \sum_{n\ge k_0} |H_n| \Gamma_i (r_n \lambda^{-n})^s\\
      &\lesssim \sum_{n\ge k_0} \lambda^{2n} (r_n\lambda^{-n})^s
      = \sum_{n\ge k_0}e^{n(2\log
        \lambda- (\alpha+ \log\lambda)s)}.
    \end{split}
  \end{equation*}
  Therefore for any $s>\frac{2\log\lambda}{\alpha+ \log\lambda},$
  we have $\mathcal{H}^s(E)< \infty$, which implies that $\hdim
  E\le \frac{2\log\lambda}{\alpha+ \log\lambda}$.

  Combining Cases 1 and 2, we have 
  \[
    \hdim E\le \min
    \Bigl\{\frac{2\log\lambda}{\alpha+ \log\lambda},
    \frac{\log\lambda}{\alpha} \Bigr\}.
    \qedhere
  \]
\end{proof}

\section{Lower bound on Hausdorff dimension}

We define the $s$-dimensional Riesz potential of a measure $\mu$
by
\[
  R_s \mu (x) = \int |x-y|^{-s} \, \mathrm{d} \mu (y).
\]
The $s$-dimensional Riesz energy of $\mu$ is
\[
  I_s (\mu) = \int R_s \mu \, \mathrm{d} \mu = \iint |x-y|^{-s}
  \, \mathrm{d} \mu (x) \, \mathrm{d} \mu (y).
\]

Let $A$ be a $2 \times 2$ integer matrix with $\det A=1$ and an
eigenvalue $\lambda>1$. Lemma~\ref{per:contained} tells us that
$\bigl\{ \bigl( \frac{m}{S_k}, \frac{j}{S_k} \bigr) : m,j=0,
\ldots, S_k-1 \bigr\}$ are periodic points with period
$n=2k+1$. After re-enumeration, we denote
%$\bigl\{ \bigl( \frac{m}{S_k}, \frac{j}{S_k} \bigr) : m,j=0,\dots, S_k-1 \bigr\}$ 
these periodic points by $\{x_{n,i} \}_{i=1}^{N_n}$, where
$N_n=S_k^2$.

For $n=2k+1$, define
\begin{equation}
  \label{en}
  E_n = \bigcup_{i=1}^{N_n} E_n^i,
\end{equation}
and recall Notation~\ref{ellipsoid} that $E_n^i\subset \E_n^i$ is the
parallelogram with centre $x_{n,i}$, and lengths of diagonal lines
$2\lambda_{n,1}$ and $2\lambda_{n,2}$. Note that $E_n\subset \bigcup_{j=1}^{|H_n|} \E_n^j$. Hence $\limsup_nE_n\subset E$.

%By Liouville's theorem on diophantine approximation , the distance between $E_n^i$ and $E_n^j$, $i\ne j$, satisfies
%\[
%  d(E_n^i, E_n^j) \ge \frac{c_1}{S_k^2|x_{n,i}-x_{n,j}|} \gtrsim \lambda^{-n},
%\]
%\textbf{where $c_1>0$ is an absolute constant.}
Now we show that the shortest distance between $E_n^i$ and $E_n^j$, $i\ne j$ is positive.  For any pair $E_n^i$ and $E_n^j$, $i\ne j$, we assume that $x_{n,i}=(\frac{i_1}{S_k},\frac{i_2}{S_k}),~x_{n,j}=(\frac{j_1}{S_k},\frac{j_2}{S_k})$, then 
\[d(E_n^i,E_n^j)\ge d(x_{n,i},l_{n,j})-2r_n\lambda_{n,2},\]
where $d(x_{n,i},l_{n,j})$ is the distance between $x_{n,i}$ and the line $l_{n,j}$ given by $f(x)=\frac{\lambda^{-1}-a}{b}(x-\frac{j_1}{S_k})+\frac{j_2}{S_k}$. When $i_1=j_1$, we have $d(E_n^i,E_n^j)\asymp S_k^{-1}$. When $i_1\ne j_1$, since $\frac{\lambda^{-1}-a}{b}$ is an algebraic number of degree 2, by Liouville's theorem on diophantine approximation, we have 
\begin{equation*}
\begin{split}
d(x_{n,i},l_{n,j})&\gtrsim\Big|\frac{\lambda^{-1}-a}{b}\Big(\frac{i_1-j_1}{S_k}\Big)+\frac{j_2-i_2}{S_k}\Big|\\
&\ge \frac{c_1}{S_k|i_1-j_1|},
\end{split}
\end{equation*}
where $c_1>0$ is a constant only depending on $A$. Hence the distance between $E_n^i$ and $E_n^j$, $i\ne j$, satisfies
%\begin{equation}\label{eq:liouville}
 % d(E_n^i, E_n^j) \gtrsim \frac{c_1}{S_k|i_1-j_1|} \gtrsim \lambda^{-n}>0.
%\end{equation}
 \begin{align}\label{eq:liouville}
d(E_n^i, E_n^j) \left\{\begin{array}{cl}
 \gtrsim \frac{1}{S_k|i_1-j_1|} \gtrsim \lambda^{-n}>0& \quad \text{if~}i_1\ne j_1,\\[2ex]
\asymp S_k^{-1}\qquad\qquad\qquad&\quad \text{if~}i_1= j_1.
 \end{array}\right.
\end{align}

For large enough $n$, we have that $\{E_n^i\}_{i=1}^{N_n}$ do
not intersect each other. 
Therefore
\[
  \L(E_n^i)= \frac{\sqrt{(a+d)^2-4}}{|bH_n|}r_n^2, \quad \L(E_n)=
  \sum_{i=1}^{N_n} \L(E_n^i)= \frac{\pi }{|b(a+d-2)|}r_n^2.
\]
%$N_n= (S_k-1)^2\asymp S_k^2$. 
%, and $B_{n,i}$ is the ball with center $x_i$ and radius $S_k^{-1}/3$. 

\subsection{Estimate when $\boldsymbol{\alpha}$ is large}

In this subsection, we give a lower bound on the Hausdorff
dimension of $E$ when $\alpha\ge\frac{1}{2} \log\lambda$.

\begin{theorem} \label{main2}
  For $\alpha\ge\frac{1}{2} \log\lambda$, let $E_n$ be defined as
  in \eqref{en}.  Then
  \[
    \limsup_{k\to\infty}E_{2k+1} \in\mathcal{G}^{s_0} (\T^2),
  \]
  where $s_0 = \min \bigl\{
  \frac{2\log\lambda}{\alpha+ \log\lambda}, \frac{\log\lambda}{\alpha}
  \bigr\}.$
\end{theorem}

Since $\limsup_{k\to\infty}E_{2k+1} \subset E$, we conclude from
Theorem~\ref{main2} that $E\in\mathcal{G}^{s_0} (\T^2)$.
%Here we need add the following condition:

%(C1) \, Assume that $r_n\le cS_k^{-1}/3$, that is, for any $1\le i\le N_n$, $$\E_n^i\subset cB_{n,i},$$ where $c>1$ is an absolute constant. This is  equivalent to ``$\alpha>1/2\log\lambda$".

%From condition (C1), the balls $\{cB_{n,i} \}_i$ don't intersect each other, and ``$cB_{n,i}:=cB(x_{n,i}, \lambda_{n,1}) \subset B(x_{n,i}, \frac{1}{2S_k})$" is  equivalent to ``$r_n\lesssim \lambda^{-n/2}$.%, where $M<1/2$."
%We have $\L(E_n)= \frac{\pi }{a+b-2}r_n^2$, and when $n=2k+1$, $N_n=S_k^2$. %$N_n= (S_k-1)^2\asymp S_k^2$. 
\begin{proof}[Proof of Theorem~\ref{main2}]
Let 
\[
  \mu_n= \frac{1}{\L(E_n)} \L|_{E_n}.
\]

Since the distribution of periodic points $\{x_{n,i} \}_i$ is very
regular, it is clear that there is a constant $C>1$ such that for
any ball $B\subset \T^2$, we have
\[
  C^{-1} \le \liminf_{k\to\infty} \frac{\mu_{2k+1} (B)}{\L(B)} \leq
  \limsup_{k\to\infty} \frac{\mu_{2k+1} (B)}{\L(B)} \le C.
\]
(In fact, we may take $C=1$.)

Now we show that $I_s(\mu_n)$ is finite for some $s>0$, here
$n=2k+1$.  Write
\begin{multline*}
  \iint\frac{1}{|x-y|^s} \, \mathrm{d} \mu_n(x) \, \mathrm{d}
  \mu_n(y)
  = I_1 + I_2 \\
  = \sum_{i=1}^{N_n} \int_{E_n^i} \int_{E_n^i}
  \frac{1}{|x-y|^s} \, \mathrm{d} \mu_n(x) \, \mathrm{d} \mu_n(y) \\
  + \sum_{i=1}^{N_n} \sum_{\substack{j=1 \\ j \ne i}}^{N_n}
  \int_{E_n^i} \int_{E_n^j} \frac{1}{|x-y|^s} \, \mathrm{d}
  \mu_n(x) \, \mathrm{d} \mu_n(y).
\end{multline*}

a) Let $\phi^s(\E)$ be the singular value function of the
ellipse $\E$. Then
\begin{equation*}
  \begin{split}
    I_1&= \sum_{i=1}^{N_n} \frac{1}{\L(E_n)^2} \int_{E_n^i}
    \int_{E_n^i} \frac{1}{|x-y|^s} \, \mathrm{d}x\, \mathrm{d}y\\
    &\lesssim \frac{1}{r_n^4}
    \sum_{i=1}^{N_n}K\frac{\L(E_n^i)^2}{\phi^s(E_n^i)}=
    \frac{1}{r_n^4}
    \sum_{i=1}^{N_n}K\frac{\pi^2r_n^4}{S_k^4\phi^s(E_n^i)} \\ 
    &=S_k^{-2} \phi^{-s} (E_n^i).
  \end{split}
\end{equation*}

b) \, Now we estimate
$I_2= \sum_{i=1}^{N_n} \sum_{\substack{j = 1 \\ j\ne i}}^{N_n}
\int_{E_n^i} \int_{E_n^j} \frac{1}{|x-y|^s} \, \mathrm{d}
\mu_n(x) \, \mathrm{d} \mu_n(y)$.

It suffices to prove that the Riesz potential
\[
  R (x)= \sum_{j\ne i_0} \int_{E_n^j} \frac{1}{|x-y|^s} \,
  \mathrm{d} \mu_n(y)
\]
is uniformly bounded for $x\in E_n^{i_0}$,
 
Since $\alpha>\frac{1}{2} \log\lambda$, for any $i\ge1$, the
number of cubes
$\{(\frac{l}{S_k}, \frac{l+1}{S_k}) \times (\frac{m}{S_k},
\frac{m+1}{S_k}) \}$ which $E_n^i$ intersects is uniformly
bounded. %Then by Liouville's theorem (on diophantineapproximation),
Hence $|i_1-j_1|\lesssim 1$, where $\frac{i_1}{S_k}$, $\frac{j_1}{S_k}$ are the first
coordinate of $x_{n,i}$ and $x_{n,j}$ respectively.  By
\eqref{eq:liouville}, the shortest distance $d_n$ between $E_n^i$
and $E_n^j$, $i\ne j$ satisfies
\[
  d_n %\ge %\frac{c_1}{S_k|i_1-j_1|} 
  \gtrsim
  \lambda^{- \frac{n}{2}}.
\]
%here $i_1$, $j_1$ are the first coordinate of $x_{n,i}$ and $x_{n,j}$ respectively.

%For $\alpha\ge \frac{1}{2} \log\lambda$, let $K_0= \sqrt{10-2\sqrt{5}}/2$, and let $M_0=2K_0^{-1}r_nS_k+1< \infty$ which is independent of $n$.  By Liouville's Theorem, the distance between $\E_n^i$ and $\E_n^j$
%\[d(\E_n^i, \E_n^j) \ge \frac{c}{S_k^2|x_i-x_j|}, \]
%We denote the shortest distance between ellipsoids $\{\E_n^i\}$ by $d_n$. 
%Then given $i\ge1$, for $j\in\{j\ge1\colon |x_j-x_i|\le M_0S_k^{-1} \}$, 
%\[d(\E_n^i, \E_n^j) \ge \frac{c}{S_kM_0} \asymp  S_k^{-1}.\]
%For $j\notin\{j\ge1\colon |x_j-x_i|\le M_0S_k^{-1} \}$, $d(\E_n^i, \E_n^j) \ge S_k^{-1}$

%Let $B_{n,i}=B(x_{n,i}, \lambda_{n,1})$, $1\le i\le N_n$.
%Since $\alpha>\frac{1}{2} \log\lambda$, there is some $c>1$, such that for any $1\le i\le N_n$, 
%\[\E_n^i\subset B_{n,i} \subset cB_{n,i} \subset B(x_{n,i}, \frac{1}{2S_k}) \]
%holds for $n$ large enough.

%For any $x\in\E_n^i$, and $y\in\E_n^j$, $i\ne j$, 
%$$|x-y|\ge d(B_{n,i},B_{n,j}),$$
%and 
%\be
%\begin{split}
%|x_{n,i}-x_{n,j}|&=d(B_{n,i},B_{n,j})+2\lambda_{n,1} \\
%&=d(B_{n,i},B_{n,j})+c^{-1} (2c\lambda_{n,1}) \\
%&\le d(B_{n,i},B_{n,j})+c^{-1}|x_{n,i}-x_{n,j}|,
%\end{split}
%\ee
%the last inequality holds since $\{cB_{n,i} \}_i$ don't intersect each other. It follows that
%\[|x-y|\ge (1-c^{-1}) |x_{n,i}-x_{n,j}|.\]
Notice that for $j\ne i_0$, $x \in E_n^{i_0}$ and $y\in E_n^j$,
  there are $0\le m\le d_n^{-1}$ and $1\le l\le d_n^{-1}$ such
  that $|x-y|\ge (m^2+l^2)^{1/2}d_n$. Then
\begin{equation*}
  \begin{split}
    I_2(x)&= \sum_{j\ne i_0} \int_{E_n^j} \frac{1}{|x-y|^s} \,
    \mathrm{d} \mu_n(y) \le \sum_{m=0}^{d_n^{-1}}
    \sum_{l=1}^{d_n^{-1}} ((m^2+l^2)^{1/2}d_n)^{-s} \mu_n(E_n^j)
    \\
    &\lesssim \frac{1}{r_n^2}S_k^s\int_0^{ d_n^{-1}} \int_0^{
      d_n^{-1}} (u^2+v^2)^{- \frac{s}{2}} \, \mathrm{d}u\,
    \mathrm{d}v\\
    &\lesssim r_n^{-2} \lambda^{\frac{1}{2}ns}
    \lambda^{\frac{1}{2}n(2-s)}=e^{n(-2\alpha+ \log\lambda)} < C_1,
  \end{split}
\end{equation*}
%Then from \eqref{munb},
%\be
%\begin{split}
%I_2&\le \sum_{i=1}^{N_n} \sum_{j=1\atop j\ne i}^{N_n} \int_{\E_n^i} \int_{\E_n^j} (1-c^{-1})^{-s}|x_{n,i}-x_{n,j}|^{-s}d\mu_n(x)d\mu_n(y) \\
%&= \sum_{i=1}^{N_n} \sum_{j=1\atop j\ne i}^{N_n} (1-c^{-1})^{-s}|x_{n,i}-x_{n,j}|^{-s} \mu_n(\E_n^i) \mu_n(\E_n^j) \\
%&\asymp \sum_{i=1}^{N_n} \frac{1}{r_n^4} \Big(\sum_{j=1\atop j\ne i}^{N_n}|x_{n,i}-x_{n,j}|^{-s} \L(\E_n^i) \L(\E_n^j) \Big) \\
%&\lesssim\sum_{i=1}^{N_n} \frac{1}{r_n^4} (r_n^2|H_n|^{-1})^2\sum_{l,m=1}^{S_k-1} ((m^2+l^2)^{1/2}S_k^{-1})^{-s} \\
%&\asymp N_nS_k^{s-4} \sum_{l,m=1}^{S_k-1} (m^2+l^2)^{-s/2} \\
%&\le S_k^{s-2} \int_1^{S_k-1} \int_1^{S_k-1} \Big((x-1)^2+ (y-1)^2\Big)^{-s/2} \\
%&\lesssim S_k^{s-2} (2(S_k-2)^2)^{1-s/2} \lesssim S_k^{s-2}S_k^{2-s} \le C',
%\end{split}
%\ee
here $C_1$ is a constant which is independent of $n$ and $x$. Then
\[
  I_2= \sum_{i=1}^{N_n} \int_{\E_n^i}I_2(x) \, \mathrm{d}
  \mu_n(x) \asymp C_1\lambda^nr_n^{-2}<C_2.
\]
where $C_2$ is a constant which is independent of $n$.

%Under condition (C1), 
Combining a) and b), we have
\[
  I_s (\mu_n) = I_1 + I_2 \le S_k^{-2} \phi^{-s} (\E_n^i)+C_2.
\]
Put
$s_0 = \min \bigl \{\frac{\log \lambda}{\alpha}, \frac{2 \log
  \lambda}{\alpha + \log \lambda} \bigr\} \leq 2$ and let $m$ be
such that $m < s_0 \le m+1$. Take $s$ such that $m < s < D$.

If $m=0$, then $s_0 = \frac{\log \lambda}{\alpha}$, and we have
\[
  \phi^{s} (E_n^i)= (\lambda_{n,1})^s=
  \frac{(\lambda^n-1)^s}{(a+d-2)^sS_k^{2s}}r_n^s.
\]
For $n$ large
enough, we then have
\begin{align*}
    S_k^{-2} \phi^{-s} (E_n^i) &\asymp
    \Bigl(\frac{(\lambda^n-1)}{\lambda^n+ \lambda^{-n} - 2}
    \Bigr)^{-s} 
    e^{\alpha sn} (\lambda^n + \lambda^{-n}-2)^{-1} \\
    &\lesssim e^{\alpha sn} \lambda^{-n} \asymp \exp \{n (\alpha
    s- \log \lambda) \}.
\end{align*}
From this, it follows since $s < \frac{1}{\alpha} \log\lambda$
that we have $\exp\{n(\alpha s- \log \lambda) \}<M$. This implies
that $S_k^{-2} \phi^{-s} (E_n^i)<M$, where $M>0$ is an absolute
constant.

If $m=1$, then
$s_0 = \frac{2 \log \lambda}{\alpha + \log \lambda}$ and
\begin{align*}
  \phi^{s} (E_n^i)
  &= \lambda_{n,1} (\lambda_{n,2})^{s-1}=
    \frac{\lambda^n-1}{(a+d-2)S_k^{2}}r_n\frac{(1-
    \lambda^{-n})^{s-1}}{(a+d-2)^{s-1}S_k^{2s-2}}r_n^{s-1} \\ 
  &= \frac{(\lambda^n-1) (1-
    \lambda^{-n})^{s-1}}{(a+d-2)^sS_k^{2s}}r_n^{s},
\end{align*}
here $\lambda_{n,2}=r_n\frac{1- \lambda^{-n}}{|H_n|}$.  For $n$
large enough, we have
\begin{align*}
  S_k^{-2} \phi^{-s} (E_n^i)
  &\asymp
    \frac{(\lambda^n+ \lambda^{-n}-2)^{s-1}}{(\lambda^n-1)
    (1- \lambda^{-n})^{s-1}} e^{\alpha sn} \\
  &\lesssim e^{\alpha sn} \lambda^{n(s-2)} = \exp\{n(\alpha
    s- (2-s) \log \lambda) \}.
\end{align*}
Note that for any $s< \frac{2\log\lambda}{\alpha+ \log\lambda}$,
we have $\exp\{n(\alpha s- (2-s) \log \lambda) \}<M$, then it
follows that $S_k^{-2} \phi^{-s} (E_n^i)<M$.

Therefore, for any
\[
  s < s_0 := \min \Bigl\{\frac{2\log\lambda}{\alpha +
    \log\lambda}, \frac{1}{\alpha} \log\lambda \Bigr\} \leq 2,
\]
we have that $I_s(\mu_{2k+1})$ is uniformly bounded for all large
enough $k$. Then from Lemma~\ref{lemma:persson}, we have that
$\limsup_kE_{2k+1} \in \mathcal{G}^{s_0} (\T^2)$.  It follows that
when $\alpha\ge\frac{1}{2} \log\lambda$,
\[
  E\in \mathcal{G}^{s_0} (\T^2).
\]
Moreover, 
\[
  \hdim E\ge \hdim \Bigl(\limsup_{k\to\infty}E_{2k+1} \Bigr) \ge
  \min\Bigl\{\frac{2\log\lambda}{\alpha+ \log\lambda},
  \frac{1}{\alpha} \log\lambda\Bigr\}.
  \qedhere
\]
%Moreover, $E\in \mathcal{G}^{s_0} (\T^2)$.
\end{proof}

\subsection{Estimate when $\boldsymbol{\alpha}$ is small}
% {\color{red}This part hasn't be finished, since the lower bound is not what we want.}

In this subsection, we give the lower bound on
$\hdim \limsup_nE_n$ when $0< \alpha< \frac{1}{2} \log\lambda$.  We
denote by $\L^1$ the Lebesgue measure restricted on $\T$.

Let $L_{x_0} := \{\, ({x_0},y) : y\in\mathbb{T} \, \}$,
${x_0}
\in\mathbb{T}$ %Without losing generality, we assume that $x= \frac{i}{S_k}$, $i\ge1$.
and
\[
  F_n(x_0) = \{\, \E_{n,i} \cap L_{x_0} : 1 \le i\le
  N_n,~\lambda_{n,2}/3\le|E_n^i\cap L_{x_0}|\le 2\lambda_{n,2}
  \, \}.
\]
Assume that $F_n(x_0) = \{\, I_{n,i} : 1\le i\le M_n \, \},$ and
$M_n= \#F_n(x_0)$.
\begin{theorem} \label{main3} For
  $0< \alpha< \frac{1}{2} \log\lambda$, given $x_0\in\mathbb{T}$,
  we have
  \[
    \limsup_{n\to\infty}F_n(x_0) \in\mathcal{G}^{s_1} (L_{x_0}),
  \]
  where $s_1= \frac{\log\lambda- \alpha}{\log\lambda+ \alpha}.$ In
  particular, for any $x_0\in\T$, we have
  \[
    \hdim \Bigl(\limsup_{n\to\infty}F_n(x_0) \Bigr) \ge s_1.
  \]
\end{theorem}

Observe that for any $x\in\mathbb{T}$, the set $F_n(x)$ may be
regarded as a subset of the intersection of $E_n$ with the line
$L_x$ on $\T^2$. Thus, applying Proposition 7.9 in
\cite{falconer}, we deduce that
\[
  \hdim \Bigl(\limsup_{n\to\infty}E_n\Bigr) \ge s_1+1=s_0.
\]
Then we get a lower bound on $\hdim E$, which coincides with
its upper bound. However, we will show more about
$\limsup_kE_{2k+1}$, namely the following.

\begin{theorem} \label{main4}
  For $0< \alpha< \frac{1}{2} \log\lambda$, we have
  \[
    \limsup_{n\to\infty}E_n\in\mathcal{G}^{s_0} (\T^2),
  \]
  where $s_0= \frac{2\log\lambda}{\log\lambda+ \alpha}.$
\end{theorem}

Since it suffices to study $\limsup_kF_{2k+1} (x_0)$ and
$\limsup_kE_{2k+1}$ instead of $\limsup_nF_{n} (x_0)$ and
$\limsup_nE_{n}$, from now on, we always assume that $n=2k+1$,
$k\ge0$.

\begin{lem} \label{num:ellipsoid}
  Given any segment with $L_l \subset L_{x_0}$, of length $\ell$,
  for large enough $n = 2k + 1$, we have
  \[
    \#\Bigl\{\, i : 1\le i\le
    N_n,~\frac{\lambda_{n,2}}{4} \le|E_n^i\cap L_l|\le
    2\lambda_{n,2} \,\Bigr\} \asymp \ell r_n\lambda^n.
  \]
  In particular, $\#M_n\asymp r_n\lambda^n$.
\end{lem}

\begin{proof}
  Assume that $L_l= \{\, (x_0,y) : y\in(y_0,y_0+ \ell) \,
  \}$. Note that $i\ge1$ with
  $\frac{\lambda_{n,2}}{4} \le|E_n^i\cap L_l|\le 2\lambda_{n,2}$
  if and only if $x_{n,i} \in P_l$, where $x_{n,i}$ is the centre
  of $\E_n^i$, and $P_l$ is the parallelogram with vertices
 % coordinates
  $(a',b')$, $(d',c')$, $(a',b'+ \ell)$,
  $(d',c'+ \ell)$, and
  \begin{align*}
   % a'&=x_0-K_0r_n\frac{\lambda^{n}-1}{|H_n|}, \qquad
      %  d'=x_0+K_0r_n\frac{\lambda^{n}-1}{|H_n|}, \\
      a'&=x_0-K_0\lambda_{n,1}, \qquad
        d'=x_0+K_0\lambda_{n,1}, \\
   % b'&=y_0-K_0'r_n\frac{\lambda^{n}-1}{|H_n|}, \qquad
     %   c'=y_0+K_0'r_n \frac{\lambda^{n}-1}{|H_n|},
    b'&=y_0-K_0' \lambda_{n,1}, \qquad
     c'=y_0+K_0'\lambda_{n,1}
  \end{align*}
  here $K_0,~K_0'>0$ are absolute constants.
  %$K_0'= \frac{\sqrt{5}-1}{2}K_0$.  
  Notice that
  \[
    \#\{\, x_i : 1\le i\le N_n,~x_i\in P_l \, \}= \#\{\,
    (p,q) \in\mathbb{Z}^2 : (p,q) \in S_kP_l \,\}.
  \]
  The coordinates of the vertices of $S_kP_l$ are
  $(a'S_k,c'S_k)$, $(d'S_k,b'S_k)$, $(a'S_k,c'S_k+ \ell S_k)$, and
  $(d'S_k,b'S_k+ \ell S_k)$.
  % \[aS_k=xS_k-K_0r_nS_k\frac{\lambda^{n}-1}{|H_n|},~dS_k=xS_k+K_0r_nS_k\frac{\lambda^{n}-1}{|H_n|}, \]
  % \[bS_k=y_0S_k-K_0'r_nS_k\frac{\lambda^{n}-1}{|H_n|},~cS_k=y_0S_k+K_0'r_nS_k\frac{\lambda^{n}-1}{|H_n|}.\]
  By Pick's theorem, we have 
  \begin{multline}
    \label{unum:ellipsoid}
    \#\{\, (p,q) \in\mathbb{Z}^2 : (p,q) \in S_kP_l \,\} \le (\lceil
    d'S_k\rceil- \lfloor a'S_k\rfloor) \lceil lS_k\rceil\\
    \le (lS_k+1) (2K_0r_nS_k\frac{\lambda^{n}-1}{|H_n|}+2) \lesssim
    \ell r_n\lambda^n,
  \end{multline}
  and 
  \begin{multline}
    \label{lnum:ellipsoid}
    \#\{\,(p,q) \in\mathbb{Z}^2 : (p,q) \in S_kP_l \,\} \ge
    (\lfloor d'S_k\rfloor- \lceil a'S_k\rceil) \lfloor
    lS_k\rfloor\\
    \ge (lS_k-1) (2K_0r_nS_k\frac{\lambda^{n}-1}{|H_n|}-2)
    \gtrsim \ell r_n\lambda^n.
  \end{multline}
  Combining \eqref{unum:ellipsoid} and \eqref{lnum:ellipsoid}, we
  have
  \[
    \#\Bigl\{\, i : 1\le i\le N_n,~\frac{\lambda_{n,2}}{4}
    \le|E_n^i\cap L_l|\le 2\lambda_{n,2} \, \Bigr\} \asymp \ell
    r_n\lambda^n.
  \]
  Take $\ell=1$. Then $M_n\asymp  r_n\lambda^n.$
\end{proof}

Now we give the proof of Theorem~\ref{main3}.

\begin{proof}[Proof of Theorem~\ref{main3}]
  Given $x_0\in\T$, in this proof, we write $F_n(x_0)$ as $F_n$
  for convenience. For $n=2k+1$, let
  \[
    \mu_n= \frac{1}{\L^1(F_n)} \L^1|_{F_n}.
  \]
  It follows from Lemma~\ref{num:ellipsoid} that
  $\L^1(F_n) \asymp r_n^{2}$.

  i) \, For any ball $B\subset L_{x_0}$, we have
  \begin{align*}
    \mu_n(B) &= r_n^{-2} \sum_{\substack{i=1 \\ I_{n,i} \cap B \ne
    \varnothing}}^{M_n} \L^1(B\cap I_{n,i}) \\
             &\lesssim r_n^{-2} \L^1 \Bigl((1+
               \frac{2\lambda_{n,2}}{r_B}) B 
               \Bigr) r_n \lambda^n r_n \lambda^{-n} \le
               \L^1\Bigl((1+ \frac{2\lambda_{n,2}}{r_B})B\Bigr).
  \end{align*}
  Since $\max_i|I_{n,i}|\to 0$, as $n\to\infty$, then
  \[
    \limsup_{n\to\infty} \frac{\mu_n(B)}{\L^1(B)} \lesssim 1.
  \]

  And for $n$ large enough, we have 
  \begin{equation*}
    \begin{split}
      \mu_n(B)&= \sum_{\substack{i=1 \\ I_{n,i} \cap B \ne
          \varnothing}}^{M_n} \frac{\L^1 (B \cap
        I_{n,i})}{\L^1(F_n)} \ge \sum_{\substack{i=1 \\ I_{n,i}
          \subset
          B}}^{M_n} \frac{\L^1(I_{n,i})}{\L^1(F_n)} \\
      &\asymp r_n^{-2} \#\{\, i : I_{n,i} \subset
      B \,\}r_n\lambda^{-n}.
    \end{split}
  \end{equation*}
  Since 
  \[
    \#\{\, i : I_{n,i} \subset B \,\} \supset \{\, i : \E_n^i\cap
    (1-2\lambda_{n,2})B\ne\varnothing \,\},
  \]
  by Lemma~\ref{num:ellipsoid}, we have 
  \[
    \#\{\, i : I_{n,i} \subset B \,\} \gtrsim
    \L^1((1-2\lambda_{n,2})B)r_n\lambda^n .
  \]
  Then 
  \[
    \mu_n(B) \gtrsim \L^1((1-2\lambda_{n,2})B),
  \]
  and it follows that for large enough $n$,
  \[
    \liminf_{n\to\infty} \frac{\mu_n(B)}{\L^1(B)} \gtrsim1.
  \]
%Given $\epsilon>0$, for $n$ large enough, we have $\max_i|A_{n,i}|< \epsilon/3$, and notice that
%\[\{A_{n,i} \colon A_{n,i} \subset B\} \supset \{A_{n,i} \colon A_{n,i} \cap (1- \epsilon)B\ne\varnothing\}.\]
%\be
%\begin{split}
%\L\Big((1- \epsilon)B\cap \bigcup_{i\in\Gamma_n}3A_{n,i} \Big)&\le \L\Big((1- \epsilon)B\cap \bigcup_{i\in\Gamma_n\atop A_{n,i} \subset B}3A_{n,i} \Big)+ \L\Big((1- \epsilon)B\cap \bigcup_{i\in\Gamma_n\atop A_{n,i} \nsubseteq B}3A_{n,i} \Big) \\
%&= \L\Big((1- \epsilon)B\cap \bigcup_{A_{n,i} \subset B}3A_{n,i} \Big) \\
%&\le \sum_{i\in\Gamma_n\atop A_{n,i} \subset B} \L(3A_{n,i})=3\sum_{i\in\Gamma_n\atop A_{n,i} \subset B} \L(A_{n,i}),
%\end{split}
%\ee
%on the other side,
%\be
%\L\Big((1- \epsilon)B\cap \bigcup_{i\in\Gamma_n}3A_{n,i} \Big)= \L((1- \epsilon)B).
%\ee
%From these inequalities, we have 
%\[\mu_n(B) \ge \frac{2}{3(1+ \sqrt{5})} \L((1- \epsilon)B).\]
  Therefore 
  \[
    C_3^{-1} \le \liminf_{n\to\infty} \frac{\mu_n(B)}{\L^1(B)} \le
    \limsup_{n\to\infty} \frac{\mu_n(B)}{\L^1(B)} \le C_3,
  \]
  where $C_3>0$ is a constant independent of $n$ and $B$. 

  2) \, Now we show that the energy of $\mu_n$ is bounded for some
  $s$. We have
  \begin{equation*}
    \begin{split}
      I_s(\mu_n)&= \iint|x-y|^{-s} \, \mathrm{d} \mu_n(x) \,
      \mathrm{d} \mu_n(y) \\
      &= \sum_{i=1}^{M_n} \sum_{j=1}^{M_n} \int_{I_{n,i}}
      \int_{I_{n,j}} |x-y|^{-s} \, \mathrm{d} \mu_n(x) \,
      \mathrm{d} \mu_n(y) \\
      &= \sum_{i=1}^{M_n} \int_{I_{n,i}} \int_{I_{n,i}}
      |x-y|^{-s} \, \mathrm{d} \mu_n(x) \, \mathrm{d} \mu_n(y) \\
      & \qquad + \sum_{i=1}^{M_n} \sum_{\substack{j=1 \\ j\ne i}}^{M_n}
      \int_{I_{n,i}} \int_{I_{n,j}} |x-y|^{-s} \,
      \mathrm{d} \mu_n(x) \, \mathrm{d} \mu_n(y).
      % &\le M_n\mu_n(I_{n,i})^2|I_{n,i}|^{1-s}+ \sum_{i=1}^{M_n}
      % \sum_{j=1\atop j\ne i}^{M_n} \int_{I_{n,i}}
      % \int_{I_{n,j}}|x-y|^{-s}d\mu_n(x)d\mu_n(y)
    \end{split}
  \end{equation*}
  1) \, When $i=j$, for $x\in I_{n,i}$ for some $i\ge1$, we have
  \[
    I_1(x) := \int_{I_{n,i}}|x-y|^{-s} \, \mathrm{d} \mu_n(y) \le
    \frac{1}{r_n^2}|I_{n,i}|^{1-s} \lesssim
    e^{n\{(1+s) \alpha- (1-s) \log\lambda\}}.
  \]
  Notice that for any
  $s< \frac{\log\lambda- \alpha}{\log\lambda+ \alpha}$, we have
  $I_1(x) \lesssim 1$, which implies that
  \begin{equation}
    \label{part1}
    I_1 := \sum_{i=1}^{M_n} \int_{I_{n,i}} \int_{I_{n,i}}
    |x-y|^{-s} \, \mathrm{d} \mu_n(x) \, \mathrm{d} \mu_n(y) \lesssim
    r_n \lambda^nr_n^{-1} \lambda^{-n}<C_4
  \end{equation}
  holds for $n$ large enough.

  2) \, For $x\in I_{n,i}$ for some $i\ge1$, we let
  \begin{equation*}
      I_2(x) := \sum_{\substack{j=1 \\ j\ne i}}^{M_n}
      \int_{I_{n,j}}|x-y|^{-s} \, \mathrm{d} \mu_n(y).\\
  \end{equation*}
  % Without losing generality, from now on we assume that
  % $x_0= \frac{m}{S_k}$ for some $m\ge1$, and $x\in I_{n,1}$.
  For $p\ge0$, let  
  \[
    A_p= \Bigl\{\, j : I_{n,j} \cap \Bigl\{\, (x_0,y) :
    \frac{p+1}{S_k}>|y-x|\ge \frac{p}{S_k} \,\Bigr\}
    \ne\varnothing \,\Bigr\}.
  \]
  Then
  \begin{align*}
    I_2(x) &\phantom{:}\leq S_1(x)+S_2(x) \\
           &:= \sum_{p=1}^{S_k-1} \sum_{j\in A_p}
             \int_{I_{n,j}} |x-y|^{-s} \, \mathrm{d} \mu_n(y) +
             \sum_{\substack{ j\in A_0 \\ j\ne 1}}
    \int_{I_{n,j}}|x-y|^{-s} \,
    \mathrm{d} \mu_n(y).
  \end{align*}
  % \[S_1(x) \colon = \sum_{p=1}^{S_k-1} \sum_{j\in A_p}
  %   \int_{I_{n,j}}|x-y|^{-s}d\mu_n(y).\]
  For $p\ge1$, $j\in A_p$, and $y\in I_{n,j}$, we have
  $|x-y|\ge \frac{p}{S_k}$. Then it follows from
  Lemma~\ref{num:ellipsoid} that for any $s\le 1$,
  \begin{equation}
    \label{s1}
    \begin{split}
      S_1(x)&= \sum_{p=1}^{S_k-1} \sum_{j\in A_p} \mu_n(I_{n,j})
      \Bigl (\frac{p}{S_k} \Bigr)^{-s} \\
      &\asymp \sum_{p=1}^{S_k-1} \#A_p p^{-s} r_n^{-1}
      \lambda^{-n} S_k^s\\
      & \asymp \lambda^{\frac{n}{2} (s-1)}
      \sum_{p=1}^{S_k-1}p^{-s} \le \lambda^{\frac{n}{2} (s-1)}
      \int_0^{S_k} u^{-s} \, \mathrm{d}u\\
      &\le \lambda^{\frac{n}{2} (s-1)} \lambda^{\frac{n}{2}
        (1-s)} < C_5.
    \end{split}
  \end{equation}
  Now we estimate
  $S_2(x) = \sum_{\substack{j\in A_0 \\ j\ne 1}}
  \int_{I_{n,j}}|x-y|^{-s} \, \mathrm{d}\mu_n(y)$.
  %By Liouville's theorem (on diophantine approximation), the distance between $E_n^i$ and $E_n^j$
%  \[
%    d(E_n^i, E_n^j) \ge \frac{c}{S_k|x_i-x_j|} \gtrsim
  %  \lambda^{-n},
 % \]
 Recall \eqref{eq:liouville}, then for $m\ne j$,
 \[
    d(E_n^m, E_n^j) %\ge \frac{c}{S_k|x_i-x_j|} 
    \gtrsim  \lambda^{-n}.
  \]
  We denote the shortest distance between intervals $\{I_{n,i} \}$
  by $d_n$, then for any $j_1\ne j_2\in A_0$, we have
  $d(I_{n,j_1},I_{n,j_2}) \ge d_n\gtrsim \lambda^{-n}$. Therefore
  \begin{equation}
    \label{s2}
    \begin{split}
      S_2(x)&\le \sum_{q=1}^{\#A_0} (qd_n)^{-s} \mu_n(I_{n,j})
      \lesssim r_n^{-2} r_n \lambda^{-n} \lambda^{ns}
      \sum_{q=1}^{\#A_0} q^{-s} \\
      &\le r_n^{-1} \lambda^{n(s-1)} \int_{0}^{\#A_0} u^{-s} \,
      \mathrm{d}u \le r_n^{-1} \lambda^{n(s-1)} (r_nS_k)^{1-s} \\
      &\lesssim r_n^{-s} \lambda^{\frac{n}{2} (s-1)} =
      e^{n\{\alpha s + \frac{1}{2} \log\lambda(s-1) \}}.
    \end{split}
  \end{equation}
  Notice that for any
  $s< \frac{\log\lambda}{\log\lambda+2\alpha}$, we have
  $S_2(x) \lesssim 1$.  From \eqref{s1} and \eqref{s2}, we have
  for any $s< \frac{\log\lambda}{\log\lambda+2\alpha}$, we have
  \begin{equation}
    \label{part2}
    I_2 := \sum_{i=1}^{M_n} \sum_{\substack{j=1 \\ j\ne i}}^{M_n}
    \int_{I_{n,i}} \int_{I_{n,j}}|x-y|^{-s} \,
    \mathrm{d} \mu_n(x) \, \mathrm{d} \mu_n(y)<C_6.
  \end{equation}
  Notice that 
  \[
    \min \Bigl\{\frac{\log\lambda - \alpha}{\log\lambda +
      \alpha}, \frac{\log \lambda}{\log \lambda + 2\alpha}
    \Bigr\} = \frac{\log\lambda- \alpha}{\log\lambda+ \alpha}
  \]
  holds for all $\alpha>0$.  Combining \eqref{part1} and
  \eqref{part2}, for any
  $s< \frac{\log\lambda- \alpha}{\log\lambda+ \alpha}$, we have
  $I_s(\mu_n)<C_4+C_6$ for $n$ large enough.

  Write $s_1= \frac{\log\lambda- \alpha}{\log\lambda+ \alpha}$.
  Applying Lemma~\ref{lemma:persson}, we have
  $\limsup_{n\to\infty}F_{n} \in\mathcal{G}^{s_1} (L_{x_0})$.
\end{proof}

Suppose that $\mu$ is a finite measure on $[0,1]^2$ which can be
disintegrated as
\[
  \mu (A) = \int \mu_x (A \cap \{x\} \times [0,1]) \, \mathrm{d}
  \nu (x) = \int \mu_x (A) \, \mathrm{d} \nu (x),
\]
for any $A\in \mathcal{B}$, where $\mu_x$ is a measure with
support in $\{x\} \times [0,1]$. We say $\{\mu_x\}_{x\in[0,1]}$
is a disintegration of $\mu$ over $\nu$. We assume moreover that
$\mu_x$ and $\nu$ are probability measures.

% For us the following lemma is usfull in with $\nu$ equal to the
% Lebesgue measure.

\begin{lemma} \label{lemma:main}
  With $\mu$ as above, we have for $x = (x_1, x_2)$ that
  \[
    R_{s+t} \mu (x) \leq \int R_s \mu_{y_1} (x_2) |x_1 -
    y_1|^{-t} \, \mathrm{d} \nu (y_1).
  \]
  In particular, if the $s$-dimensional Riesz potentials of
  $\mu_x$ are uniformly bounded and if $\nu$ has a bounded
  $t$-dimensional potential, then
  \[
    I_{s+t} (\mu) \leq \sup_x \lVert R_s \mu_x \rVert_\infty
    \lVert R_t \nu \rVert_\infty.
  \]
\end{lemma}

\begin{proof}
  Writing $x = (x_1, x_2)$ and $y = (y_1, y_2)$, we have
  \begin{align*}
    R_{s+t} \mu (x) & = \iint |(x_1, x_2) - (y_1, y_2) |^{-s-t} \,
                      \mathrm{d} \mu_{y_1} (y_2) \, \mathrm{d} \nu
                      (y_1) \\
                    & \leq \int \biggl( \int |x_2 - y_2|^{-s} \,
                      \mathrm{d} \mu_{y_1} (y_2) \biggr) |x_1 -
                      y_1|^{-t} \, \mathrm{d} \nu (y_1) \\
                    & = \int R_s \mu_{y_1} (x_2) |x_1 - y_1|^{-t}
                      \, \mathrm{d} \nu (y_1).
  \end{align*}
  Hence
  \[
    \lVert R_{s+t} \mu \rVert_\infty \leq \sup_x \lVert R_s \mu_x
    \rVert_\infty \lVert R_t \nu \rVert_\infty.
  \]
  Since $\mu$ is a probability measure, we have
  $I_{s+t} (\mu) \leq \lVert R_{s+t} \mu \rVert_\infty$.
\end{proof}

% Recall that if $A\subset \mathbb{R}^d$ and
% $B \subset \mathbb{R}^t$, $d,t\in\mathbb{N}$, the product
% $A \times B$ is defined as the set of points with first
% coordinate in $A$ and second coordinate in $B$, that is,
%\[A \times B= \{(x, y) \in  \mathbb{R}^{d+t} \colon x \in A ,~ y \in B\}.\]

%\begin{theorem}[Product formula\cite{falconer}]
%If $A\subset \mathbb{R}^d$, $B \subset \mathbb{R}^t$ are Borel sets, then
%\[\dim_{\rm H} (A \times B) \ge \dim_{\rm H}A + \dim_{\rm H}B.\]
%\end{theorem}
%{\color{blue}
%\begin{lemma} \label{falconer}[Corollary 7.10]
%Let $F$ be a Borel subset of $\mathbb{R}^2$. Then, for almost all $x$ (in the sense of 1-dimensional
%Lebesgue measure), 
%\[\dim_{\rm H} (F \cap L_x) \le \max\{0, \dim_{\rm H}F - 1\}.\]
%\end{lemma}

%}
Now we finish the proof of our main result Theorem~\ref{main}.

\begin{proof}[Proof of Theorem~\ref{main}]
  For $\alpha\ge\frac{1}{2} \log\lambda$, from Theorem~\ref{main2}, we have
  \[
    E\in\mathcal{G}^{s_0} (\T^2),
  \]
  where
  $s_0= \min\Bigl\{\frac{2\log \lambda}{\alpha + \log\lambda},
  \frac{1}{\alpha} \log\lambda\Bigr\}$, hence
  \[
    \hdim E\ge s_0.
  \]

  For $0< \alpha< \frac{1}{2}
  \log\lambda$, %since $\limsup_{k\to\infty}E_{2k+1}= (\limsup_{k\to\infty}F_{2k+1} (x)) \times \{x\in\mathbb{T} \}$,
  Let $\tilde{\mu}_n $ be defined as
  \[
    \tilde{\mu}_n(A)= \int\mu_{x,n} (A) \, \mathrm{d}x,
  \]
  for $A\in\mathcal{B}$, where
  $\mu_{x,n}= \frac{1}{\L^1(F_n(x))} \L^1|_{F_n(x)}$. For $n\ge1$,
  \[
    \tilde{\mu}_n (\T^2\setminus E_n)=0,
  \]
  that is, the support of $\tilde{\mu}_n $ is in $E_n$.  Then
  $\{\mu_{x,n} \}_{x\in\T}$ is a disintegration of
  $\tilde{\mu}_n$ over Lebesgue measure $\L^1$.
 
  By Theorem~\ref{main3}, the $s_1$-dimensional Riesz potentials
  of $\mu_{2k+1}$ are uniformly bounded for $k$ large enough, and
  note that $\L^1$ has a bounded $(1-\varepsilon)$-dimensional
  potential for any $\varepsilon > 0$. Applying
  Theorem~\ref{main3} and Lemma~\ref{lemma:main}, for $k$ large
  enough, we have
  \begin{equation*} \label{alphaless}
    I_{s_1+1-\varepsilon} (\tilde{\mu}_{2k+1})< \infty,
    % \dim_{\rm H} (\limsup_{k\to\infty}E_{2k+1}) \ge
    % s_1+1= \frac{2\log\lambda}{\log\lambda+ \alpha}.
  \end{equation*}
  here $s_0 - \varepsilon = s_1 + 1 -\varepsilon$ when
  $\alpha< \frac{1}{2} \log\lambda$. It is easy to show that for
  any ball $B\subset \T^2$,
  \[
    1\lesssim\liminf_{k\to\infty} \frac{\tilde{\mu}_{2k+1} (B)}{\L
      (B)} \le \limsup_{k \to \infty} \frac{\tilde{\mu}_{2k+1}
      (B)}{\L (B)} \lesssim 1.
  \]
  Applying Lemma~\ref{lemma:persson}, we have
  $\limsup_{k\to\infty}E_{2k+1} \in\mathcal{G}^{s_0 -
    \varepsilon} (\T^2)$. Since $\varepsilon > 0$ is arbitrary,
  we have
  $\limsup_{k\to\infty}E_{2k+1} \in\mathcal{G}^{s_0}$ and in
  particular,
  % It follows from \eqref{alphaless} and Theorem~\ref{main2} that for $\alpha\ge0$,
  $\hdim E\ge s_0.$
  
  Therefore, for $\alpha >0$, we have 
  \[
    E\in\mathcal{G}^{s_0} (\T^2),
  \]
  and $\hdim E\ge s_0$.  Combining with
  Lemma~\ref{upperdimension}, we have $\hdim E= s_0$.
\end{proof}

\subsection*{Acknowledgements}
This research of Zhangnan Hu was supported by China Scholarship Council.

\end{document}